\documentclass[leqno]{amsart}
\usepackage[normal]{boilerart}
\usepackage{mathrsfs}
\usepackage{tikz-cd}

\newcommand{\angs}[1]{\langle #1 \rangle}
\newcommand{\twa}{\widetilde{\mathscr{O}}}

\def\reflectedop#1#2{\mathop{\reflectbox{$#1{#2}$}}}

\title{On the fibrewise effective Burnside $\infty$-category}

\author{Clark Barwick}
\address{Department of Mathematics, Massachusetts Institute of Technology, 77 Massachusetts Avenue, Cambridge, MA 02139-4307, USA}
\email{clarkbar@math.mit.edu}

\author{Saul Glasman}
\address{Department of Mathematics, Massachusetts Institute of Technology, 77 Massachusetts Avenue, Cambridge, MA 02139-4307, USA}
\email{sglasman@math.mit.edu}

\begin{document}

\begin{abstract} Effective Burnside $\infty$-categories, introduced in \cite{M1}, are the centerpiece of the $\infty$-categorical approach to equivariant stable homotopy theory. In this \'etude, we recall the construction of the twisted arrow $\infty$-category, and we give a new proof that it is an $\infty$-category, using an extremely helpful modification of an argument due to Joyal--Tierney \cite{JT}. The twisted arrow $\infty$-category is in turn used to construct the effective Burnside $\infty$-category. We employ a variation on this theme to construct a fibrewise effective Burnside $\infty$-category. To show that this constuctionworks fibrewise, we introduce a fragment of a theory of what we call \emph{marbled simplicial sets}, and we use a yet further modified form of the Joyal--Tierney argument.
\end{abstract}

\maketitle

\section{The twisted arrow $\infty$-category} There are three basic endofunctors of the simplex category $\Delta$: the identity $\id$, the opposite $\op$ (which simply reverses the ordering on the objects), and the constant functor $\kappa$ at $\mathbf{[0]}$. There is also the associative \emph{join} or \emph{concatenation} operation $\star\colon\fromto{\Delta\times\Delta}{\Delta}$, so that $\mathbf{[m]}\star\mathbf{[n]}=\mathbf{[m+n+1]}$. This join operation gives rise to a semigroup structure $\star$ on the set $\End(\Delta)$ of endomorphisms, so that $(f\star g)\mathbf{[m]}=f(\mathbf{[m]})\star g(\mathbf{[m]})$. Velcheva \cite{} shows that the semigroup $\End(\Delta)$ is freely generated by $\id$, $\op$, and $\kappa$.

Of particular import to us will be the endofunctor $\varepsilon\coloneq\op\star\id$. This induces a functor $\widetilde{\mathscr{O}}\coloneq\varepsilon^{\star}\colon\fromto{s\Set}{s\Set}$, so that
\[\widetilde{\mathscr{O}}(X)_n=X(\mathbf{[n]}^{\op}\star\mathbf{[n]})=X_{2n+1}.\]
This functor is (a twisted form of) the \emph{edgewise subdivision.}

Lurie proved the following in \cite[Pr. 4.2.3]{DAGX}, but, as a way of introducing the basic tools we will use here, we shall give our own, appreciably simpler, proof. 

\begin{prp}\label{prp:twarriscat} For any $\infty$-category $C$, the functor
\[\fromto{\widetilde{\mathscr{O}}(C)}{C^{\op}\times C}\]
induced by the inclusions $\into{\op}{\op\star\id}$ and $\into{\id}{\op\star\id}$ is a left fibration. In particular, $\widetilde{\mathscr{O}}(C)$ is an $\infty$-category.
\end{prp}

\noindent The idea of our argument is to adapt an argument introduced by Joyal--Tierney \cite{JT}. Here is the key notion.

\begin{dfn} A class of monomorphisms $E$ in an ordinary category \emph{satisfies the right cancellation property} if for any monomorphisms $u\colon\fromto{x}{y}$ and $v\colon\fromto{y}{z}$, if $v\circ u$ and $u$ both lie in $E$, then so does $v$.
\end{dfn}

\begin{exm} Observe that in any model category in which the cofibrations are precisely the monomorphisms, the trivial cofibrations satisfy the right cancellation property.
\end{exm}

\begin{rec} Let 
\[s_n\colon\into{I^n\coloneq\Delta^{\{0,1\}} \cup^{\Delta^{\{1\}}} \cdots \cup^{\Delta^{\{n-1\}}} \Delta^{\{n-1,n\}}}{\Delta^n}\]
be the inclusion of the spine of the $n$-simplex; this is of course inner anodyne. More generally, if $K=\{a_0,\dots,a_k\}$ is a nonempty totally ordered finite set, then write
\[I^K\coloneq\Delta^{\{a_0,a_1\}} \cup^{\Delta^{\{a_1\}}} \cdots \cup^{\Delta^{\{a_{k-1}\}}} \Delta^{\{a_{k-1},a_k\}}\subset\Delta^K.\]
\end{rec}

The maps $s_n$ also \emph{determine} the class of inner anodyne maps the following sense:
\begin{lem}[Joyal--Tierney, \protect{\cite[Lm. 3.5]{JT}}]\label{lem:JT35} A saturated class of monomorphisms of simplicial sets that satisfies the right cancellation property contains the inner anodyne maps if and only if it contains the spine inclusions $s_n$ for $n\geq 2$.
\end{lem}

For the proof of Pr. \ref{prp:twarriscat}, we will need a version of this statement that is suitable for left fibrations.

\begin{lem}\label{lem:JT35leftfibs} A saturated class of monomorphisms of simplicial sets that satisfies the right cancellation property contains the left anodyne maps if and only if it contains the spine inclusions $s_n$ for $n\geq 2$ as well as the horn inclusions
\[i_1\colon\into{\Lambda_0^1}{\Delta^1}\text{\quad and\quad}i_2\colon\into{\Lambda_0^2}{\Delta^2}.\]
\begin{proof} Suppose $E$ is such a class. Let $J^n$ denote the union of edges in $\Delta^n$
\[\Delta^{\{0,1\}}\cup\Delta^{\{0, 2\}}\cup \bigcup_{i = 2}^{n - 1}\Delta^{\{i, i + 1\}}.\]
First we claim that the inclusion $J^n \to \Delta^n$ belongs to $E$. Indeed, the inclusion
\[\into{J^n}{\Delta^2 \cup^{\Delta^{\{2\}}} \Delta^{\{2,\dots,n\}}}\]
is clearly in $E$, as are the inclusions
\[\into{I^n  \cup^{\Delta^{\{0, 1\}}} \Delta^{\{0, 1\}}}{\Delta^2 \cup^{\Delta^{\{2\}}}\Delta^{\{2,\dots,n\}}}\]
and
\[\into{I^n  \cup^{\Delta^{\{0, 1\}}} \Delta^{\{0, 1\}}}{\Delta^n},\]
which proves the claim.

The remaining necessity is that the inclusion
\[\into{J^n}{\Lambda^n_0}\]
lie in $E$. Following the proof of Lm. \ref{lem:JT35}, we'll prove something slightly more general. Write $\Delta^{\hat{s}}$ for the face $\Delta^{\{0,1,\dots,s-1,s+1,\dots,n\}}$ of $\Delta^n$ opposite $s$, and for any subset $S\subset\{0,\dots,n\}$, write
\[\Lambda^n_S\coloneq\bigcup_{s\notin S}\Delta^{\hat{s}}.\]
(Equivalently, $\Lambda^n_S$ is the union of the faces of $\Delta^n$ that contain the simplex $\Delta^S$.) We shall now show that the inclusion
\[\into{J^n}{\Lambda^n_S} \]
is in $E$ for any $S$ with
\[\{0\} \subseteq S \subsetneqq \{0, 2, 3, \cdots, n\}. \]
This prescription on $S$ implies that $\Delta^{\{0, 1\}}$ is an edge of $\Lambda^n_S$, so this definition makes sense. We'll use induction on both $n$ and $n - |S|$. Of course, the statement is vacuous if $n = 1$. Suppose that $n - |S| = 1$, which is the least possible value, so that $S = \{0, 2, \cdots n\} \setminus \{a\}$ for some $a$ with $2 \leq a \leq n$. Then 
\[[\into{J^n}{\Delta^{\{0, 1\}} \cup \Delta^{\hat{1}}}] \in E,\]
and since 
\[\Delta^{\hat{a}} \cap (\Delta^{\{0, 1\}} \cup \Delta^{\hat{1}}) = \Delta^{\{0, 1\}} \cup (\Delta^{\hat{1}} \cap \Delta^{\hat{a}}),\]
we see that $[\into{J^n}{\Lambda^n_S}] \in E$ in this case.

In general, choose some $a \notin S$ with $a \neq 1$. Then we're reduced to showing that
\[[\into{\Lambda^n_{S \cup \{a\}}}{\Lambda^n_S}] \in E,\]
which we'll naturally accomplish by showing that 
\[[\into{(\Delta^{\hat{a}} \cap \Lambda^n_{S \cup \{a\}}}{\Delta^{\hat{a}})}] \in E.\]
But since $\{0\} \subseteq S \subsetneqq (\{0, 2, 3, \cdots n\} - \{a\})$, this follows from the induction hypothesis.
\end{proof}
\end{lem}

\begin{proof}[Proof of Pr. \ref{prp:twarriscat}] Write $\varepsilon_!$ for the left Kan extension $\fromto{s\Set}{s\Set}$ along $\varepsilon$. This is left adjoint to $\varepsilon^{\star}$. Now consider the class $E$ of monomorphisms $f\colon\fromto{X}{Y}$ of simplicial sets such that the map
\[\fromto{\varepsilon_!(X)\cup^{X^{\op}\sqcup X}(Y^{\op}\sqcup Y)}{\varepsilon_!(Y)}\]
is a trivial cofibration for the Joyal model structure. It's easy to see that $E$ is a saturated class that satisfies the right cancellation property. Furthermore, by adjunction, it's clear that any morphism of $E$ has the left lifting property with respect to $\fromto{\widetilde{\mathscr{O}}(C)}{C^{\op}\times C}$. Consequently, Lm. \ref{lem:JT35leftfibs} implies that we need only to show that the spine inclusions $s_n$ and the horn inclusion $i_2$ all lie in $E$.

If $n\geq 2$, write $\{\overline{n},\overline{n-1},\dots,\overline{0}\}$ for the poset $\categ{[n]}^{\op}$. Observe that the monomorphism
\[\fromto{\varepsilon_!(I^n)\cup^{I^{n,\op}\sqcup I^n}(\Delta^{n,\op}\sqcup\Delta^n)}{\varepsilon_!(\Delta^n)}\]
is isomorphic to the inclusion of the iterated union
\[
U\coloneq(\cdots((\Delta^{\{\overline{n},\overline{n-1},\dots,\overline{0}\}}\cup^{\Delta^{\{\overline{0},0\}}}\Delta^{\{0,\dots,n-1,n\}})\cup^{I^{\{\overline{1},\overline{0},0,1\}}}\Delta^{\{\overline{1},\overline{0},0,1\}})\cup^{I^{\{\overline{2},\overline{1},1,2\}}}\cdots)\cup^{I^{\{\overline{n},\overline{n-1},n-1,n\}}}\Delta^{\{\overline{n},\overline{n-1},n-1,n\}}
\]
into $\Delta^{\{\overline{n},\overline{n-1},\dots,\overline{0},0,\dots,n-1,n\}}$. It's a simple matter to see that the inclusion
\[\into{\Delta^{\{\overline{n},\overline{n-1},\dots,\overline{0}\}}\cup^{\Delta^{\{\overline{0},0\}}}\Delta^{\{0,\dots,n-1,n\}}}{\Delta^{\{\overline{n},\overline{n-1},\dots,\overline{0},0,\dots,n-1,n\}}}\]
is inner anodyne, and the inclusion
\[\into{\Delta^{\{\overline{n},\overline{n-1},\dots,\overline{0}\}}\cup^{\Delta^{\{\overline{0},0\}}}\Delta^{\{0,\dots,n-1,n\}}}{U}\]
is clearly an iterated pushout of inner anodyne maps, so the right cancellation property implies that $\into{U}{\Delta^{\{\overline{n},\overline{n-1},\dots,\overline{0},0,\dots,n-1,n\}}}$ is a trivial cofibration for the Joyal model structure, whence $s_n$ lies in $E$.

It remains to show that the horn inclusions $i_1$ and $i_2$ lie in $E$. First, note that the monomorphism
\[\fromto{\varepsilon_!(\Lambda^1_0)\cup^{(\Lambda^1_0)^{\op}\sqcup\Lambda^1_0}(\Delta^{1,\op}\sqcup\Delta^1)}{\varepsilon_!(\Delta^1)}\]
is isomorphic to the spine inclusion $s_3\colon\fromto{I^3}{\Delta^3}$, which is clearly inner anodyne; hence $i_1$ lies in $E$. Observe also that the monomorphism
\[\fromto{\varepsilon_!(\Lambda^2_0)\cup^{(\Lambda^2_0)^{\op}\sqcup\Lambda^2_0}(\Delta^{2,\op}\sqcup\Delta^2)}{\varepsilon_!(\Delta^2)}\]
is isomorphic to the inclusion of the union
\[V\coloneq\Delta^{2,\op}\cup^{(\Lambda^2_0)^{\op}}(\Delta^{\{\overline{2},\overline{0},0,2\}}\cup^{\Delta^{\{\overline{0},0\}}}\Delta^{\{\overline{1},\overline{0},0,1\}})\cup^{\Lambda^2_0}\Delta^2\]
into $\Delta^{\{\overline{2},\overline{1},\overline{0},0,1,2\}}$. The simplical set $V$ contains the spine $I^{\{\overline{2},\overline{1},\overline{0},0,1,2\}}$, and it's a simple matter to see that the inclusion $\into{I^{\{\overline{2},\overline{1},\overline{0},0,1,2\}}}{V}$ is inner anodyne; hence by the right cancellation property, we deduce that $\into{V}{\Delta^{\{\overline{2},\overline{1},\overline{0},0,1,2\}}}$ is a trivial cofibration for the Joyal model structure. It thus follows that $i_2$ lies in $E$.
\end{proof}

We call $\widetilde{\mathscr{O}}(X)$ the \emph{twisted arrow $\infty$-category of $X$}. We justify this language by noting that if $X$ is a $1$-category, then $\widetilde{\mathscr{O}}(X)$ is a $1$-category as well, and it agrees with the classical, $1$-categorical twisted arrow category.

\section{The effective Burnside $\infty$-category} The functor $\varepsilon$ also induces a functor $\varepsilon_{\star}\colon\fromto{s\Set}{s\Set}$, which is right adjoint to $\varepsilon^{\star}$. Consequently, for any simplicial set $X$,
\[(\varepsilon_{\star}X)_n\cong\Mor(\widetilde{\mathscr{O}}(\Delta^n),X)\]

\begin{dfn} If $C$ admits all pullbacks, then we define the \emph{effective Burnside $\infty$-category} of $C$ is the simplicial subset
\[A^{\eff}(C)\subset(\varepsilon_{\star}(C^{\op}))^{\op}\]
whose $n$-simplices are those functors $X\colon\fromto{\widetilde{\mathscr{O}}(\Delta^n)^{\op}}{C}$ such that for any integers $0\leq i\leq k\leq l\leq j\leq n$, the square
\begin{equation*}
\begin{tikzpicture} 
\matrix(m)[matrix of math nodes, 
row sep=4ex, column sep=4ex, 
text height=1.5ex, text depth=0.25ex] 
{X_{ij}&X_{kj}\\ 
X_{il}&X_{kl}\\}; 
\path[>=stealth,->,font=\scriptsize] 
(m-1-1) edge (m-1-2) 
edge (m-2-1) 
(m-1-2) edge (m-2-2) 
(m-2-1) edge (m-2-2); 
\end{tikzpicture}
\end{equation*}
is a pullback.
\end{dfn}

The name is justified by the following result.
\begin{prp}[\protect{\cite[Pr. 5.6]{M1}}] If $C$ is an $\infty$-category that admits all pullbacks, then $A^{\eff}(C)$ is an $\infty$-category.
\end{prp}

\noindent We will generalize this result by providing a fibrewise effective Burnside construction in the next section. But first, let us discuss a form of the effective Burnside $\infty$-category in which the maps that appear are from certain chosen classes.

\begin{dfn} A \emph{triple} $(C,C_{\dag},C^{\dag})$ of $\infty$-categories consists of an $\infty$-category $C$ and two subcategories $C_{\dag}\subset C$ and $C^{\dag}\subset C$, each of which contains all the equivalences.\footnote{Recall \cite[\S 1.2.11]{HTT} that subcategories determine and are determined by subcategories of their homotopy categories.} The morphisms of $C_{\dag}$ are called \emph{ingressive}, and the morphisms of $C^{\dag}$ are called \emph{egressive}.

A triple $(C,C_{\dag},C^{\dag})$ is said to be \emph{adequate} if, for any ingressive morphism $\cofto{Y}{X}$ and any egressive morphism $\fibto{X'}{X}$, there exists a pullback square
\begin{equation*}
\begin{tikzpicture} 
\matrix(m)[matrix of math nodes, 
row sep=4ex, column sep=4ex, 
text height=1.5ex, text depth=0.25ex] 
{Y'&X'\\ 
Y&X\\}; 
\path[>=stealth,->,font=\scriptsize] 
(m-1-1) edge (m-1-2) 
edge (m-2-1) 
(m-1-2) edge[->>] (m-2-2) 
(m-2-1) edge[>->] (m-2-2); 
\end{tikzpicture}
\end{equation*}
in which $\fromto{Y'}{X'}$ is ingressive, and $\fromto{Y'}{Y}$ is egressive. (Such a square will be called \emph{ambigressive}.)

The \emph{effective Burnside $\infty$-category} of an adequate triple $(C,C_{\dag},C^{\dag})$ is the simplicial subset
\[A^{\eff}(C,C_{\dag},C^{\dag})\subset(\varepsilon_{\star}(C^{\op}))^{\op}\]
whose $n$-simplices are those functors $X\colon\fromto{\widetilde{\mathscr{O}}(\Delta^n)^{\op}}{C}$ such that for any integers $0\leq i\leq k\leq l\leq j\leq n$, the square
\begin{equation*}
\begin{tikzpicture} 
\matrix(m)[matrix of math nodes, 
row sep=4ex, column sep=4ex, 
text height=1.5ex, text depth=0.25ex] 
{X_{ij}&X_{kj}\\ 
X_{il}&X_{kl}\\}; 
\path[>=stealth,->,font=\scriptsize] 
(m-1-1) edge[>->] (m-1-2) 
edge[->>] (m-2-1) 
(m-1-2) edge[->>] (m-2-2) 
(m-2-1) edge[>->] (m-2-2); 
\end{tikzpicture}
\end{equation*}
is an ambigressive pullback.
\end{dfn}

\begin{thm}[\protect{\cite[Th. 12.2]{M1}}]\label{thm:pain} Suppose $(C,C_{\dag},C^{\dag})$ and $(D,D_{\dag},D^{\dag})$ adequate triples, and suppose $p\colon\fromto{C}{D}$ an inner fibration that preserves ingressive morphisms, egressive morphisms, and ambigressive pullbacks. Then the induced functor
\[A^{\eff}(p)\colon\fromto{A^{\eff}(C,C_{\dag},C^{\dag})}{A^{\eff}(D,D_{\dag},D^{\dag})}\]
is an inner fibration as well. Furthermore, assume the following.
\begin{enumerate}[(\ref{thm:pain}.1)]
\item For any ingressive morphism $g\colon\cofto{s}{t}$ of $D$ and any object $x\in C_s$, there exists an ingressive morphism $f\colon\cofto{x}{y}$ of $C$ covering $g$ that is both $p$-cocartesian and $p_{\dag}$-cocartesian.
\item Suppose $\sigma$ a commutative square
\begin{equation*}
\begin{tikzpicture} 
\matrix(m)[matrix of math nodes, 
row sep=4ex, column sep=4ex, 
text height=1.5ex, text depth=0.25ex] 
{x'&y'\\ 
x&y,\\}; 
\path[>=stealth,->,font=\scriptsize] 
(m-1-1) edge[>->] node[above]{$f'$} (m-1-2) 
edge[->>] node[left]{$\phi$} (m-2-1) 
(m-1-2) edge[->] node[right]{$\psi$} (m-2-2) 
(m-2-1) edge[>->] node[below]{$f$} (m-2-2); 
\end{tikzpicture}
\end{equation*}
of $C$ such that the square $p(\sigma)$ is an ambigressive pullback in $D$, the morphism $f'$ is ingressive, the morphism $\phi$ is egressive, and the morphism $f$ is $p$-cocartesian. Then $f'$ is $p$-cocartesian if and only if the square is an ambigressive pullback (and in particular $\psi$ is egressive).
\end{enumerate}
Then an edge $f\colon\fromto{y}{z}$ of $A^{\eff}(C,C_{\dag},C^{\dag})$ is $A^{\eff}(p)$-cocartesian if it is represented as a span
\begin{equation*}
\begin{tikzpicture}[baseline]
\matrix(m)[matrix of math nodes, 
row sep=3ex, column sep=3ex, 
text height=1.5ex, text depth=0.25ex] 
{&u&\\ 
y&&z,\\}; 
\path[>=stealth,->,font=\scriptsize,inner sep=1.5pt] 
(m-1-2) edge[->>] node[above left]{$\phi$} (m-2-1) 
edge[>->] node[above right]{$\psi$} (m-2-3); 
\end{tikzpicture}
\end{equation*}
in which $\phi$ is egressive and $p$-cartesian and $\psi$ is ingressive and $p$-cocartesian.
\end{thm}

\begin{nul} Observe that the projections
\[\fromto{\widetilde{\mathscr{O}}(\Delta^n)^{\op}}{\Delta^n}\text{\quad and\quad}\fromto{\widetilde{\mathscr{O}}(\Delta^n)^{\op}}{(\Delta^n)^{\op}}\]
induce inclusions
\[\into{C_{\dag}}{A^{\eff}(C,C_{\dag},C^{\dag})}\textrm{\quad and\quad}\into{(C^{\dag})^{\op}}{A^{\eff}(C,C_{\dag},C^{\dag})}.\]
\end{nul}

\begin{cnstr}\label{cnstr:duals} Suppose $S$ an $\infty$-category, and suppose $p\colon\fromto{X}{S}$ an inner fibration. Declare a morphism of $X$ to be ingressive if it is lies over an equivalence of $S$, and declare a morphism of $X$ to be egressive if it is $p$-cartesian. Then the morphism of triples
\[\fromto{(X,X_{\dag},X^{\dag})}{(S,\iota S,S)}\]
satisfies all the conditions of Th. \ref{thm:pain}, whence one has an inner fibration 
\[A^{\eff}(p)\colon\fromto{A^{\eff}(X,X_{\dag},X^{\dag})}{A^{\eff}(S,\iota S,S)}\]
We may now pull this inner fibration back along the equivalence $\into{S^{\op}}{A^{\eff}(S,\iota S,S)}$ to obtain an inner fibration
\[p^{\vee}\colon\fromto{X^{\vee}}{S^{\op}}.\]
This will be called the \emph{right dual of $p$}. The objects of $X^{\vee}$ are the objects of $X$, but an edge $\fromto{x}{y}$ is a span
\begin{equation*}\label{eq:mapfromxtoyinXdual}
\begin{tikzpicture}[baseline]
\matrix(m)[matrix of math nodes, 
row sep=2ex, column sep=3ex, 
text height=1.5ex, text depth=0.25ex] 
{&u&\\ 
x&&y\\}; 
\path[>=stealth,->,inner sep=0.9pt,font=\scriptsize] 
(m-1-2) edge node[above left]{$f$} (m-2-1) 
edge node[above right]{$g$} (m-2-3); 
\end{tikzpicture}
\end{equation*}
of $X$ in which $f$ is a $p$-cartesian edge, and $p(g)$ is a degenerate edge of $S$. This morphism is $p^{\vee}$-cocartesian just in case $g$ is an equivalence.

One can equally well form the \emph{left dual of $p$}, which is the inner fibration
\[((p^{\op})^{\vee})^{\op}\colon\fromto{((X^{\op})^{\vee})^{\op}}{S^{\op}},\]
which, to distinguish it from the right dual, we denote by ${}_{\vee}p\colon\fromto{{}_{\vee}X}{S^{\op}}$. In ${}_{\vee}X$, the objects are again those of $X$, but an edge $\fromto{x}{y}$ is a cospan
\begin{equation*}
\begin{tikzpicture} 
\matrix(m)[matrix of math nodes, 
row sep=2ex, column sep=3ex, 
text height=1.5ex, text depth=0.25ex] 
{&u&\\ 
x&&y\\}; 
\path[>=stealth,<-,inner sep=0.9pt,font=\scriptsize] 
(m-1-2) edge node[above left]{$f$} (m-2-1) 
edge node[above right]{$g$} (m-2-3); 
\end{tikzpicture}
\end{equation*}
of $X$ in which $p(f)$ is a degenerate edge of $S$, and $g$ is $p$-cocartesian. This morphism is ${}_{\vee}p$-cartesian just in case $f$ is an equivalence.

One also has the two opposite duals
\[(p^{\op})^{\vee}=({}_{\vee}p)^{\op}\text{\quad and\quad}(p^{\vee})^{\op}={}_{\vee}(p^{\op}).\]

It is shown in \cite{BGN} that if $p$ is a cartesian fibration classified by a functor $F\colon\fromto{S^{\op}}{\Cat_{\infty}}$, then $p^{\vee}$ is a cocartesian fibration classified by $F$, and of course the opposite dual $(p^{\vee})^{\op}={}_{\vee}(p^{\op})$ is a cartesian fibration classified by $\op\circ F$. Dually, if $p$ is a cocartesian fibration classified by a functor $G\colon\fromto{S}{\Cat_{\infty}}$, then ${}_{\vee}p$ is a cocartesian fibration classified by $G$, and the opposite dual $(p^{\op})^{\vee}=({}_{\vee}p)^{\op}$ is a cocartesian fibration classified by $\op\circ G$.
\end{cnstr}

\section{The fibrewise effective Burnside $\infty$-category} Let $p\colon X \to S$ be a cocartesian fibration of $\infty$-categories in which each fiber admits pullbacks and all the pushforward functors preserve pullbacks. Then the straightening of $p$ is a functor
\[F\colon S \to \Cat_\infty\]
which factors through the subcategory $\Cat_\infty^{\textit{pb}}$ of $\infty$-categories admitting pullbacks and pullback-preserving functors. The effective Burnside category construction defines a functor 
\[A^{\eff}\colon\Cat_\infty^{\textit{pb}} \to \Cat_\infty,\]
and by unstraightening the composite $A^{\eff} \circ F$, we get a cocartesian fibration $q\colon A^{\eff}_S(X) \to S$ such that for any vertex $s \in S$,
\[q^{-1}(s) \simeq A^{\eff}(X_s).\]
Our goal in the next part of this appendix is to to provide a direct construction of $A^{\eff}_S(X)$. The structural support for this will be a homotopy theory of ``marbled simplicial sets," a tiny fragment of an as-yet-unknown generalization of Lurie's theory of categorical patterns \cite[Appendix B]{HA}.

\begin{dfn}
A \emph{marbled simplicial set} is a triple $(S,M,B)$ consisting of a simplicial set $S$ together with
\begin{itemize}
\item a collection $M\subset S_1$ of edges of $S$ -- whose elements will be called the \emph{marked edges} -- that contains all the degenerate edges, and
\item a collection $B\subset\Mor(\Delta^1 \times \Delta^1,S)$ of squares -- whose elements will be called the \emph{blazed squares} -- that contains all constant squares.
\end{itemize}

The category of marbled simplicial sets and maps that preserve the marked edges and the blazed squares will be denoted $s\Set^{\textit{mbl}}$. 
\end{dfn}

\begin{exm} For any simplicial set $S$, we obtain a marbled simplicial set $S^{\sharp \flat}$ in which all edges are marked but only the constant squares are blazed. We will abuse notation slightly and write $s\Set^{\textit{mbl}}_{/S}$ for the category $s\Set^{\textit{mbl}}_{/S^{\sharp \flat}}$.
\end{exm}

\begin{exm} Suppose $p\colon\fromto{X}{S}$ a cocartesian fibration whose fibers $X_s$ all admit pullbacks and whose pushforward functors $\fromto{X_s}{X_t}$ preserve pullbacks. Then one obtains a marbled simplicial set $X^{\natural\natural}$ in which the marked edges are precisely the $p$-cocartesian edges, and the blazed squares are precisely the pullback squares which are contained in the fibers of $p$.
\end{exm}

\begin{dfn} Suppose $p\colon E \to B$ is a morphism of marbled simplicial sets. Then $p$ is called a \emph{marbled fibration} if it is of the form $\fromto{X^{\natural\natural}}{S^{\sharp \flat}}$ for some cocartesian fibration $\fromto{X}{S}$ whose fibers $X_s$ all admit pullbacks and whose pushforward functors $\fromto{X_s}{X_t}$ preserve pullbacks.\footnote{One could define fibrations over a more general marbled base, but we will not need this generality here.}
\end{dfn}

\begin{dfn} An inclusion $i\colon\into{K}{L}$ of marbled simplicial sets is a \emph{marbled trivial cofibration} if for any marbled fibration $p\colon E \to B$ and any solid arrow square
\begin{equation*}
\begin{tikzpicture} 
\matrix(m)[matrix of math nodes, 
row sep=4ex, column sep=4ex, 
text height=1.5ex, text depth=0.25ex] 
{K&E\\ 
L&B,\\}; 
\path[>=stealth,->,font=\scriptsize] 
(m-1-1) edge (m-1-2) 
edge node[left]{$i$} (m-2-1) 
(m-1-2) edge node[right]{$p$} (m-2-2) 
(m-2-1) edge (m-2-2)
(m-2-1) edge[dotted,inner sep=0.5] (m-1-2); 
\end{tikzpicture}
\end{equation*}
a dotted lift exists.
\end{dfn}

\begin{nul} It is natural to expect that, for any simplicial set $S$, there is a model structure on $s\Set^{\textit{mbl}}_{/S}$ whose fibrant objects are the marbled fibrations $\fromto{X^{\natural\natural}}{S^{\sharp\flat}}$ and whose cofibrations are the monomorphisms. We leave such questions to enterprising readers.
\end{nul}

\begin{dfn} Recall that $s\Set^+$ denotes the category of marked simplicial sets. Let
\[F \colon s\Set^+ \to s\Set^{\textit{mbl}}\]
be the unique functor such that
\begin{itemize}
\item $F((\Delta^n)^\flat)$ is the full subcategory of $\twa(\Delta^n)^{\op} \times \Delta^n$ spanned by those triples $((i, j), h)$ for which $0  \leq i \leq j \leq h \leq n$, in which
\begin{itemize}
\item an edge is marked just in case its image in $\twa(\Delta^n)^{\op}$ is constant, and
\item a square is blazed just in case it's spanned by vertices
\[((i_0, j_0), h), ((i_0, j_1), h), ((i_1, j_0), h), ((i_1, j_1), h)\]
where $0 \leq i_1 \leq i_0 \leq j_0 \leq j_1 \leq h \leq n$;
\end{itemize}
\item $F((\Delta^1)^\sharp)$ has the same underlying blazed simplicial set as $F((\Delta^1)^\flat)$, but has all edges marked;
\item $F$ commutes with all colimits.
\end{itemize}

There is, in addition, an obvious natural transformation $\fromto{F}{(-)^{\flat}}$, where $(-)^{\flat}$ is the functor that carries any marked simplicial set to the marbled simplicial set with the same markings in which only the constant squares are blazed. On simplices, the natural transformation is the restriction of the projection $\fromto{\twa(\Delta^n)^{\op}\times\Delta^n}{\Delta^n}$.
\end{dfn}

\begin{nul} Clearly $F((\Delta^0)^\flat)$ is simply $(\Delta^0)^{\sharp\flat}$. The marbled simplicial set $F((\Delta^1)^\flat)$ is the nerve of the category
\begin{equation*}
\begin{tikzpicture}[baseline]
\matrix(m)[matrix of math nodes,
row sep=3ex, column sep=4ex,
text height=1.5ex, text depth=0.25ex]
{000 & & \\
001 & & 111. \\
& 011 & \\};
\path[>=stealth,->,font=\scriptsize]
(m-1-1) edge node[left]{$\sim$} (m-2-1)
(m-3-2) edge (m-2-1)
edge (m-2-3);
\end{tikzpicture}
\end{equation*}
in which (in addition to the degenerate ones) the edge labeled by $\sim$ is marked, and no nonconstant squares is blazed. The marbled simplicial set $F((\Delta^2)^\flat)$ is the nerve of the category
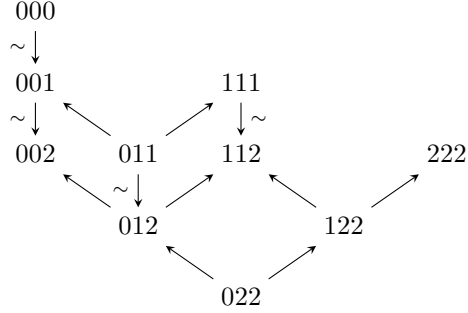
\begin{equation*}
\begin{tikzpicture}[baseline]
\matrix(m)[matrix of math nodes,
row sep=3ex, column sep=4ex,
text height=1.5ex, text depth=0.25ex]
{000 & & \\
001 & & 111 \\
002 & 011 & 112 & & 222 \\
& 012 & & 122 \\
& & 022 & \\};
\path[>=stealth,->,font=\scriptsize]
(m-1-1) edge node[left]{$\sim$} (m-2-1)
(m-2-1) edge node[left]{$\sim$} (m-3-1)
(m-3-2) edge (m-2-1)
edge node[left]{$\sim$} (m-4-2)
edge (m-2-3)
(m-2-3) edge node[right]{$\sim$} (m-3-3)
(m-4-2) edge (m-3-1)
edge (m-3-3)
(m-4-4) edge (m-3-3)
edge (m-3-5)
(m-5-3) edge (m-4-2)
edge (m-4-4);
\end{tikzpicture}
\end{equation*}
in which (in addition to the degenerate ones) the edges labeled by $\sim$ are marked, and (in addition to the constant ones) the square
\begin{equation*}
\begin{tikzpicture}[baseline]
\matrix(m)[matrix of math nodes,
row sep=4ex, column sep=4ex,
text height=1.5ex, text depth=0.25ex]
{022 & 012 \\
012 & 112 \\};
\path[>=stealth,->,font=\scriptsize]
(m-1-1) edge node[above]{} (m-1-2)
edge node[left]{} (m-2-1)
(m-1-2) edge node[right]{} (m-2-2)
(m-2-1) edge node[below]{} (m-2-2);
\end{tikzpicture}
\end{equation*}
is blazed.
\end{nul}

\begin{dfn} Suppose $p\colon\fromto{X}{S}$ a cocartesian fibration whose fibers $X_s$ all admit pullbacks and whose pushforward functors $\fromto{X_s}{X_t}$ preserve pullbacks. We define $A^{\eff}_S(X)$ to be the unique marked simplicial set over $S^{\sharp}$ yielding, for any marked map $\sigma\colon\fromto{K}{S^{\sharp}}$, a bijection
\[\Hom_{s\Set^+_{/S}}(K, A^{\eff}_S(X)) \cong \Hom_{s\Set^{\textit{mbl}}_{/S}}(F(K),X),\]
natural in $\sigma$.
\end{dfn}

If $s\in S_0$, then an $n$-simplex of the fiber $A^{\eff}_S(X)_s$ is a functor $F(\Delta^n)\to X_s$ taking all marked edges to equivalences and all blazed squares to pullback squares. There's an obvious map
\[A^{\eff}_S(X)_s \to A^{\eff}(X_s)\]
given by restriction to $\twa(\Delta^n)^{\op} \times \Delta^{\{n\}}$, and it's a simple matter to see that this map is a trivial Kan fibration. This means that the projection $\rho\colon A^{\eff}_S(X) \to S$ has the desired fibers. What's not clear at this point is whether $\rho$ is an inner fibration or anything like that. In fact, what's true is the following:
\begin{thm} The functor $\rho\colon A^{\eff}_S(X) \to S$ is a cocartesian fibration whose marked edges are precisely the cocartesian edges.
\end{thm}

The following key lemma isolates most of what we need about the combinatorics of the functor $F$.
\begin{lem} \label{lem:spinemarb} Let 
\[s_n\colon\into{I^{n,\flat} = (\Delta^{\{0,1\}} \cup^{\Delta^{\{1\}}} \cdots \cup^{\Delta^{\{n-1\}}} \Delta^{\{n-1,n\}})^{\flat}}{(\Delta^n)^{\flat}}\]
be the inclusion of the spine of the $n$-simplex. Then $F(s_n)$ is a marbled trivial cofibration.
\begin{proof} We induct on $n$. For $n = 1$, the statement is vacuous, so we are reduced to showing that the inclusion
\[w_n\colon\into{F(\Delta^{\{0, \cdots, n - 1\}}) \cup F(\Delta^{\{n - 1, n\}})}{F(\Delta^n)}\]
is a marbled trivial cofibration. We'll simply factor $w_n$ into a composite of a few maps, each of which is clearly a marbled trivial cofibration, as follows. For a collection of objects $J$ of $F(\Delta^n)$, we'll denote the full subcategory spanned by $J$ by $\angs{J}$. All marblings are inherited from $F(\Delta^n)$ in the following factorization:
\begin{equation*}
\begin{tikzpicture}[baseline]
\matrix(m)[matrix of math nodes,
row sep=4ex, column sep=4ex,
text height=1.5ex, text depth=0.25ex]
{F(\Delta^{\{0, \cdots, n - 1\}}) \cup F(\Delta^{\{n - 1, n\}})\\
\angs{\{((i, j), h) \, | \, i < n - 1 \wedge j < n\}} \cup
\angs{\{((n - 2, n-1), n-1), ((n-1, n- 1), n-1)\}} \cup
F(\Delta^{\{n - 1, n\}})\\
\angs{\{((i, j), h) \, | \, j < n\}} \cup
\angs{\{((i, j), h) \, | \, n -1 \leq i \wedge j\leq n \wedge h \leq n\}}\\
F(\Delta^n).\\};
\path[>=stealth,->,font=\scriptsize]
(m-1-1) edge[right hook->] (m-2-1)
(m-2-1) edge[right hook->] (m-3-1)
(m-3-1) edge[right hook->] (m-4-1);
\end{tikzpicture}
\end{equation*}
It is easy to see that each of these is a marbled trivial cofibration.
\end{proof}
\end{lem}

\begin{ntn} If $P$ is any simplicial subset of $\Delta^n$, then we'll denote by $l P$ the following marked simplicial set:
\begin{itemize}
\item if $P$ does not contain the edge $\Delta^{\{0, 1\}}$, then $lP=P^{\flat}$;
\item if $P$ does not contain the edge $\Delta^{\{0, 1\}}$, then $lP=(P,M)$, where $M=\{\Delta^{\{0, 1\}}\}\cup s_0(P_0)$.
\end{itemize}
\end{ntn}

\begin{lem} The functor $\rho\colon A^{\eff}_B(T) \to B$ is an inner fibration.
\begin{proof} The class of monomorphisms $f\colon X \to Y$ of simplicial sets such that $F(f^\flat)$ is a marbled trivial cofibration is a saturated class of morphisms with the right cancellation property. By Lm. \ref{lem:spinemarb} and the observation above, it contains all the inner anodyne maps.
\end{proof}
\end{lem} 

To prove that $\rho$ is a cocartesian fibration, we note that there's certainly a sufficient supply of marked edges in $A^{\eff}_B(T)$, so if we can show that marked edges are cocartesian, then $\rho$ will be a cocartesian fibration. To this end, we first note that the marked anodyne lefn horn inclusions
\[i_1\colon l \Lambda^1_0 \to l \Delta^1\text{\quad and\quad}i_2\colon l \Lambda^2_0 \to l \Delta^2 \]
have the property that $F(i_1)$ and $F(i_2)$ are marbled trivial cofibrations.

Now the desired result follows directly from the following, which is an adaptation of Lms. \ref{lem:JT35} and \ref{lem:JT35leftfibs} for the cocartesian model structure.
\begin{lem} The smallest saturated class $E$ of morphisms of \emph{marked} simplicial sets with the right cancellation property and containing the marked spine inclusions $s_n$ for $n\geq 2$ and the marked left horn inclusions $i_1$ and $i_2$ also contains all left horn inclusions
\[i_n\colon l \Lambda^n_0 \to l \Delta^n.\]
for $n \geq 2$.
\begin{proof} The proof is almost exactly the same as that of Lm. \ref{lem:JT35leftfibs}. First we note that the inclusion $l J^n \to l \Delta^n$ belongs to $E$; the argument is exactly as in Lm. \ref{lem:JT35leftfibs}, except that all simplicial sets are marked via $l$. Furthermore, the inclusion
\[\into{l J^n}{l \Lambda^n_S} \]
lies in $E$ for any $S$ with
\[\{0\} \subseteq S \subsetneqq \{0, 2, 3, \cdots, n\}, \]
again with the argument of Lm. \ref{lem:JT35leftfibs} modified only to mark all simplicial sets via $l$.
\end{proof}
\end{lem}


\bibliographystyle{amsplain}
\bibliography{Gcats}

\end{document}